\documentclass[a4paper,12pt]{article}

\usepackage[latin1]{inputenc}
\usepackage[T1]{fontenc}

\usepackage{appendix}
\usepackage{array}
\usepackage{placeins}

\usepackage[fleqn]{amsmath}

\usepackage{amsthm}
\usepackage{amsfonts}
\usepackage{empheq}
\usepackage{proof}
\usepackage{subfigure}
\usepackage{float}
\usepackage{graphicx}
\usepackage{exscale, relsize}
\usepackage{hyphenat}
\usepackage{algorithmic}
\usepackage{algorithm}
\usepackage{listings}
\usepackage{url}
\usepackage{multirow}
\usepackage{enumitem}
\usepackage{tikz}

\floatstyle{ruled}
\newfloat{algorithmfloat}{thp}{ficheiro}
\floatname{algorithmfloat}{Algoritmo}

\theoremstyle{plain} 
\newtheorem{theorem}{Theorem}[section]

\theoremstyle{definition} 
\newtheorem{definition}{Definition}[section]

\theoremstyle{remark}

\title {Forward Backward Stochastic Differential Equations - Asymptotics and a Large Deviations Principle}

\author{Ana Bela Cruzeiro \\
 Andr\'{e} de Oliveira Gomes}
\date{}

\begin{document} 
\pagenumbering{arabic}
\numberwithin{equation}{section}
\setlength{\parindent}{1 cm}

\maketitle
\begin{abstract}

  We study  the asymptotic behaviour of solutions of  Forward Backward Stochastic Differential Equations  in the coupled case, when the diffusion coefficient of the forward equation is multiplicatively perturbed by a small parameter that converges to zero. Furthermore, we establish a Large Deviation Principle for the laws of the corresponding processes.

\end{abstract}

\section{Introduction}
Forward-Backward Stochastic Differential Equations (FBSDEs for short) is a subject in Stochastic Analysis that is being  increasingly  studied by the ma-\\thematical community in the recent years.

Initially FBSDEs become known due to Stochastic Optimal Control Theory, with the pionering work [3], and, after  this kind of stochastic equations were developed in various fields, such as Mathematical Finance or Ma- \\ thematical Physics. The connections between Stochastic Differential Equations and Partial Differential Equations are well known since the celebrated Feynman-Kac formula. In this perspective, Peng [17] realized that a solution of a Backward Stochastic Differential Equation could be used as a probabilistic interpretation for the solutions of a huge class of quasilinear parabolic PDEs.

In this work, we present  the asymptotic study and the establishment of a Large Deviations Principle for the laws of the solutions of the Forward Backward Stochastic Differential Equations in the fully coupled case. One of the main points is   the independence of PDEs results in the asymptotic study of the FBSDE. These results only use probabilistic arguments. They are local in time, although,  it is also possible to establish global results in time, as we remark in Section 5. In this case the methods rely on PDE results.
 We present the asymptotic study of FBSDEs solutions when the diffusion coefficient of the forward equation approaches zero. This question adresses the problem of the convergence of classical/viscosity solutions of the quasilinear parabolic system of PDEs associated to the Backward Equation. When this quasilinear parabolic system of PDEs takes the form of the Backward Burguers Equation, the problem becomes the convergence of the solution when the viscosity parameter converges to zero. This study is the object of Section 3. 

In Section 4, using the classical tools of the Freidlin-Wentzell Theory and a Simple Version of the "Contraction Principle" (cf [7]), a Large Deviations Principle for the laws of the corresponding Forward Backward Processes of the system  is established.

\section{Preliminaries}
Let $T > 0$ , $d, k \in \mathbb{N}$ and $(\Omega, \mathcal{F}, \{\mathcal{F}_t \}_{0 \leq t \leq T}, \mathbb{P})$ be a complete filtered probability space. On such a space  we consider  a $d$- dimensional Brownian Motion $(B_t)_{0 \leq t \leq T}$ such that $\{  \mathcal{F}_t \}_{0 \leq t \leq T}$ is the natural filtration of $(B_t)_{0 \leq t \leq T}$, augmented with the collection of the null sets of $\Omega$, $\mathcal{N}:= \{ A \subset \Omega : \exists \, G \in \mathcal{F} \, $such that $ \mathbb{P}(G)=0\}$.

We define the following functional spaces, given $t \in [0,T]$:

$ S_t ^2(\mathbb{R}^d) $, the space of the continuous $ \{\mathcal{F}_t \}_{t \leq s \leq T} $ adapted processes $\phi : \Omega \times [t,T] \longrightarrow\ \mathbb{R}^d$ such that  \\

$\| \phi \|_{\infty}^2 = \mathbb{E} \displaystyle \sup_{t \leq s \leq T} {\mid \phi_s \mid}^2 \ <\infty$

\begin{flushleft}
and $H_t ^2 (\mathbb{R}^{k \times d} ) $, the space of  the $ \{\mathcal{F}_t \}_{t \leq s \leq T} $ adapted measurable processes $\phi : \Omega \times [t,T] \longrightarrow\ \mathbb{R}^{k \times d}$ such that  

\end{flushleft} 

$\| \phi \|_{2}^2 = \mathbb{E} \displaystyle \int_t^T { \mid \phi_s \mid}^2 \mathrm{d} s$. \\

These spaces are naturally complete normed spaces. 

The goal of this work is to study the assymptotic behaviour when $\varepsilon \rightarrow 0$ of the Forward Backward Stochastic Differential Equations (FBSDE for short) 
\vspace{-10 pt}
\begin{equation}
\begin{cases}
X_s^{t, \varepsilon, x} = x + \displaystyle \int_t^s f(r, X_r^{t, \varepsilon,x}, Y_r^{t, \varepsilon, x} ) \mathrm{d}r + \sqrt{\varepsilon} \displaystyle \int_t^s \sigma (r, X_r^{t, \varepsilon, x}, Y_r^{t, \varepsilon} ) \mathrm{d}Br \\
Y_s^{t, \varepsilon, x} =   h(X_T^{t,x, \varepsilon}) + \displaystyle \int_s^T g(r, X_r^{t, \varepsilon, x}, Y_r^{t, \varepsilon, x}, Z_r^{t, \varepsilon , x }) \mathrm{d}r -  \displaystyle \int_s^T Z_r^{t, \varepsilon, x }   \mathrm{d}Br \\
x \,\, \in \, \mathbb{R}^d  ;\,\,\, 0 \leq t \leq s \leq T \
\end {cases}
\end{equation}

\vspace{1em} 

\begin{flushleft}
where 
\end{flushleft}

$ f: [0,T] \times \mathbb{R}^d \times \mathbb{R}^k \rightarrow \mathbb{R}^d $

$g: [0, T] \times \mathbb{R}^d \times \mathbb{R}^k \times \mathbb{R}^{k \times d} \rightarrow \mathbb{R}^k$ 

$\sigma : [0,T] \times \mathbb{R}^d \times \mathbb{R}^k \rightarrow \mathbb{R}^{d \times d}$ 

$h: \mathbb{R}^k \rightarrow \mathbb{R}^k $
\newline 
are measurable functions with respect to the  Borelian fields in the Euclidean spaces where they are defined. If we assume the classical regularity assumptions on the coefficients of the FBSDE (that we present in section 2), theorem 1.1 of [5]  ensures that (2.1) has a unique local solution in $\mathcal{M}_t = S_t ^2(\mathbb{R}^d)  \times S_t ^2(\mathbb{R}^k) \times H_t ^2 (\mathbb{R}^{k \times d} )$. The space $\mathcal{M}_t$  is a Banach Space for the natural norm associated to the product topology structure.\newline
 The solution will be denoted   $(X_s^{t, \varepsilon, x}, Y_s^{t, \varepsilon, x},  Z_s^{t, \varepsilon, x} )_{t \leq s \leq T}$, or sometimes only  $(X_s^\varepsilon, Y_s^\varepsilon, Z_s^\varepsilon)_{t \leq s \leq T}$ for simplicity.\\
 In [15] it is shown, with probabilistic techniques, that $u^\varepsilon (t, x)= Y_t^{t,\varepsilon, x}$, which is a deterministic vector by\textit{ Blumenthal's 0-1 Law} ( remark 1.2 of [5] ), is a viscosity solution of the associated quasilinear parabolic partial differential equation 
  \vspace{-5 pt}
  \begin{equation}
 \begin{cases}
 \dfrac{\partial (u^\varepsilon)^l}{\partial t} (t,x) + \displaystyle \dfrac{\varepsilon}{2} \displaystyle \sum_{i=1}^d a_{ij} (t,x, u^\varepsilon (t,x)) \dfrac{\partial^2  (u^\varepsilon)^l}{\partial x_i \partial x_j} (t,x) +\\
  \displaystyle \sum_{i=1}^d f_i^l (t,x, u^\varepsilon (t,x)) \dfrac{\partial (u^\varepsilon)^l}{\partial x_i} (t,x) + \\
 g^l (t,x, u^\varepsilon (t,x), \sqrt{\varepsilon} \nabla_x u^\varepsilon (t,x) \sigma^\varepsilon (t,x, u^\varepsilon (t,x))) = 0 \\
 u^\varepsilon (T,x) = h (x) \\
  \,\, x \in \, \mathbb{R}^d ; \,\, t \in [0,T]; \,\, l=1,...,k
 \end{cases}
 \end{equation}
where $a_{i,j}= (\sigma \sigma^T)_{i,j}$. 

We define below the notion of  viscosity solution for the parabolic system of Partial Differential Equations (PDEs for short) (2.2).

For each $\varepsilon >0$,  consider the following differential operator,
\begin{align*}
 &(L_{l}^\varepsilon\varphi) (t,x,y,z) = \dfrac{\varepsilon}{2}  \displaystyle \sum_{i,j=1}^d a_{ij} (t,x,y) \dfrac{\partial^2 \varphi^l }{\partial x_i \partial x_j} (t,x)
+<f(t,x,y), \nabla \varphi^l (t,x) >  \\
&l=1,...,k \,\,\forall \varphi \in C^{1,2} ( [0,T] \times \mathbb{R}^d, \mathbb{R}^k),\,t\in [0,T],\,(x,y,z) \in \mathbb{R}^d \times \mathbb{R}^k \times \mathbb{R}^{k \times d}
\end{align*}

The space $C^{1,2} ( [0,T] \times \mathbb{R}^d, \mathbb{R}^k)$ is the space of the functions $\phi : [0,T]\times \mathbb{R}^d \rightarrow \mathbb{R}^k$, which are  $C^1$ with respect to the first variable and $C^2$ with respect to the second variable.

The system (2.2) reads 

\begin{equation}
\begin{cases}
\dfrac{\partial (u^\varepsilon)^l}{\partial t } + (L_l^\varepsilon u^{\varepsilon}) (t,x, u^{\varepsilon} (t,x), \nabla_x u^{\varepsilon} (t,x) \sigma (t,x, u^{\varepsilon} (t,x) ) )+ \\
g^l (t,x,u^{\varepsilon} (t,x), \sqrt{\varepsilon} \nabla_x u^{\varepsilon} (t,x) \sigma (t,x, u^{\varepsilon} (t,x)) ) = 0 \\
u^{\varepsilon} (T,x)= h(x), \,\, x \in \mathbb{R}^d, \,\,\, t \in [0,T], l=1,...,k
\end{cases}
\end{equation}

\begin{definition} \textbf{Viscosity solution}\\
Let $u^\varepsilon \in C \big ( [0,T]\times \mathbb{R}^d, \mathbb{R}^k \big )$. The function  $u^\varepsilon$ is said to be a viscosity sub-solution (resp. super-solution) of the system (2.2) if

\begin{center}
$(u^\varepsilon)^l (T,x) \leq h^l (x); \,\, \forall l=1,...,k; \,\, x \in \mathbb{R}^d$
\end{center} 

\begin{center}
( resp. $(u^\varepsilon)^l (T,x) \geq h^l (x); \,\, \forall l=1,...,k; \,\, x \in \mathbb{R}^d$ )
\end{center} 

\begin{flushleft}
and for each $l=1,...,k, (t,x) \in [0,T]\times \mathbb{R}^d, \varphi \in C^{1,2} ( [0,T] \times \mathbb{R}^d, \mathbb{R}^k) $ such that $(t,x)$ is a local minimum (resp. maximum) point of $\varphi - u^\varepsilon$, we have\\
\end{flushleft}

\begin{center}
$\dfrac{\partial \varphi}{\partial t} (t,x) + (L^{\varepsilon}_l \varphi) (t,x,u^\varepsilon (t,x), \sqrt{\varepsilon} \nabla_x \varphi (t,x) \sigma (t,x,u^\varepsilon (t,x))) +  \vspace{2 pt} 
 g^l (t,x, u^\varepsilon (t,x), \sqrt{\varepsilon} \nabla_x \varphi (t,x) \sigma (t,x,u^\varepsilon (t,x)))\geq 0; l=1,...,k$\\
 (resp. $\leq 0).$
\end{center}

The function $u^\varepsilon$ is said to be a viscosity solution of the system (2.2) if $u^\varepsilon$ is both a viscosity sub-solution and a viscosity super-solution of this system.
\end{definition}

Under more restrictive assumptions (that we present in Section 3), we shall have 
\begin{equation}
u^\varepsilon (t,x) = Y_t^{t, \varepsilon, x}
\end{equation}
and $u^\varepsilon$ will be  actually a classical solution of (2.2). This can be proved using the \textit{Four Step Scheme Methodology} of \textit{Ma-Protter-Young} [14] with the help of \textit{Ladyzhenskaja's} work in quasilinear parabolic PDEs [12].
If $f(s,x,y) = y$,  $k=d$, and $a_{i,j}= (\sigma \sigma^T)_{i,j} = I_{d \times d}$,  the system (2.2) becomes a Backward Burgers Equation 

\begin{equation}
 	\begin{cases}
 		\dfrac{\partial u^{\varepsilon } }{\partial t} + \dfrac{\varepsilon}{2} \Delta u^{\varepsilon}+ u^{\varepsilon} \nabla u^{\varepsilon} + g^\varepsilon = 0
 		\\
 		u^\varepsilon(T,x) = h(x) \\
 		 x \in \mathbb{R}^d
 	\end{cases}
 \end{equation}
 
 If $u^\varepsilon$ solves the problem above, $v^\varepsilon (t,x) = - u^\varepsilon (T-t, x)$ solves a Burgers Equation
 
 \begin{equation}
 	\begin{cases}
 		\dfrac{\partial v^{ \varepsilon } }{\partial t} - \dfrac{\varepsilon}{2} \Delta v^{\varepsilon}+ v^{\varepsilon} \nabla v^{\varepsilon} + g^\varepsilon = 0
 		\\
 		v^\varepsilon(0,x) = -h(x) \\
 		 x \in \mathbb{R}^d
 	\end{cases}
 \end{equation}
 which is an important simplified model for turbulence and describes the motion of a compressible fluid with a viscosity parameter $\dfrac{\varepsilon}{2}$ under the influence of an external force $g^\varepsilon$. \\
 The reference [21] is an example in the deterministic PDE literature where the behaviour of the above evolution problem when $\varepsilon\rightarrow0$ is studied. It is known (cf [14]) that when $G$ is an open  bounded set of $\mathbb{R}^3$ with a smooth boundary $\partial G$, if $g \, \in L^{\infty}(0,T; L(G))$, there exists $u^\varepsilon$ unique solution of a Burgers forward equation in the intersection of the functional spaces $L^{\infty}(0,T; L(G)) \cap\, L^2(0,T; W_0^{1,2}(G))$. When $g$ is in $ L^{\infty}(0,T;W_0^{3,2}(G))$, then $u^\varepsilon$ will converge in $L^2(G)$ to the corresponding solution of the limiting equation on a small non-empty time interval as the viscosity parameter goes to zero [21].\\
 One of our intentions is to study probabilistically (2.2), using the connection with the  Forward Backward Stochastic Differential Equation (2.1), with the link provided via the identification
 \begin{equation}
 u^\varepsilon (t,x) = Y_t^{t, \varepsilon, x}
 \end{equation}
 In [9] and [18], it is shown that a Large Deviation Principle  holds for the couple $(X_s^{t, \varepsilon, x}, Y_s^{t, \varepsilon, x})_{t \leq s \leq T}$ in the case where the FBSDE (2.1) is decoupled, namely, when $f = f(s,x)$ and $\sigma = \sigma (s,x)$.
 But the fundamental a-posteriori property 
 \begin{equation}
 Y_s^{t, \varepsilon, x} (\omega) = u^\varepsilon (s, X_s^{t, \varepsilon, x} (\omega) ) , s \in [t,T] \,\, a.s
 \end{equation}
(see corollary 1.5 of [5]) 
turns  the forward equation of (2.1) in an equation which only depends on $(X_s^{t, \varepsilon, x})_{t \leq s \leq T}$. With the work [1] on the so-called Freidlin-Wentzell Theory, which deals with diffusions given by
\begin{equation}
\begin{cases}
dX_t^{\varepsilon} = b^\varepsilon (t, X_t^{\varepsilon}) dt + \sqrt{\varepsilon} \sigma^\varepsilon (t, X_t^{\varepsilon}) dB_t \\
X_0^\varepsilon = X_0
\end{cases}
\end{equation}
 with both the drift and the diffusion coefficient depending on $\varepsilon$ , we are able to show that (2.1) obeys a Large Deviation Principle even in the coupled case, thus generalizing the results presented in [18].  We mention the work [10] where the authors do the asymptotic study of FBSDES with generalized Burgers type nonlinearities.
 
 The main interest of this method is that we do not use PDE results to study the asymptotic behaviour of the Forward Backward Stochastic Differential System (2.1). In particular we have the possibility to study with probabilistic techniques the behaviour of a Burgers type equation when the viscosity parameter converges to zero.  The other important point is   the establishment of a Large Deviation Principle to the laws of the Forward and Backward Processes, even when the equations are coupled, unlike [9] and [18].
  
\section{The Asymptotic Behaviour when $\varepsilon \rightarrow 0$}

In order to establish our results, we make the following set of assumptions.
\begin{flushleft}
 \textbf{Assumption A}
\end{flushleft}

We say $f,g, \sigma, $ and $h$ satisfy \textit{(A.A)} if there exists two constants $L, \Lambda >0$ such that:
\begin{flushleft}
\textit{(A.1)}
  $\forall\, t \in [0,T]$, $\forall \, (x,y,z), (\overline{x}, \overline{y}, \overline{z}) \in \mathbb{R}^d \times \mathbb{R}^k \times \mathbb{R}^{k \times d}$:
\end{flushleft}
  
\vspace{-20 pt} 
 
\begin{flushleft}
$ \mid f(t,x,y) - f(t,\overline{x}, \overline{y} )  \mid \leq L ( \mid x - \overline{x} \mid + \mid y - \overline{y} \mid )$ 

\vspace{3 pt}

 $\mid g(t,x,y,z) - g(t,\overline{x},\overline{y}, \overline{z}) \mid \leq L (\mid x - \overline{x} \mid + \mid y - \overline{y} \mid  + \mid z - \overline{z} \mid  )$ 
 
 \vspace{3 pt}
 
 $\mid h(x) - h(\overline{x}) \mid \leq L (\mid x - \overline{x} \mid) $ \\

\vspace{3 pt}

 $\mid \sigma (t,x,y) - \sigma (t,\overline{x}, \overline{y}) \mid  \leq L (\mid x - \overline{x} \mid  + \mid y - \overline{y} \mid   ) $ 
 \end{flushleft}

 \vspace{1 pt}
 
 \begin{flushleft}
\textit{(A.2)}
  $\forall\, t \in [0,T]$, $\forall \, (x,y,z), (\overline{x}, \overline{y}, \overline{z}) \in \mathbb{R}^d \times \mathbb{R}^k \times \mathbb{R}^{k \times d}$:
\end{flushleft}
  
  \vspace{-20 pt} 
  
 \begin{flushleft}
 < $\textit{x} - \overline{x},  f(t,x,y) - f(t,\overline{x},y) > \leq L \mid x - \overline{x} \mid^2$\\
 
  \vspace{3 pt}
 
< $\textit{y} - \overline{y},  f(t,x,y) - f(t,x, \overline{y} ) > \leq L \mid y - \overline{y} \mid^2$\\
 \end{flushleft}
 
 \vspace{1pt}
 
 \begin{flushleft}
\textit{(A.3)}
  $\forall\, t \in [0,T]$, $\forall \, (x,y,z)\in \mathbb{R}^d \times \mathbb{R}^k \times \mathbb{R}^{k \times d}$:
\end{flushleft}
  
  \vspace{-20 pt} 
  
  \begin{flushleft}
$ \mid f(t,x,y) \mid \leq \Lambda (1 + \mid x \mid + \mid y \mid )$ \\
$\mid g(t,x,y,z) \mid \leq \Lambda (1 + \mid x \mid + \mid y \mid + \mid z \mid) $\\
$\mid \sigma (t,x,y) \mid \leq \Lambda (1 + \mid x \mid + \mid y \mid)$\\
$\mid h(x) \mid \leq \Lambda (1 + \mid x \mid ) $\\
  \end{flushleft}
  
   \vspace{1 pt}

\begin{flushleft}
\textit{(A.4)}
  $\forall\, t \in [0,T]$, $\forall \, (x,y,z)\in \mathbb{R}^d \times \mathbb{R}^k \times \mathbb{R}^{k \times d}$:
\end{flushleft}  

\vspace{-20 pt}

\begin{flushleft}
$u \mapsto f (t, u, y) $\\
$v \mapsto g(t,x,v,z)$ \,\, are\,\, continuous \,\, mappings.
\end{flushleft}

Under this set of hypothesis, \textit{theorem 1.1} of [5] ensures that there exists a constant $C=C(L) >0$ , depending only on $L$, such that for every $T \leq C$ \textit{(1.1)} admits a unique solution in $\mathcal{M}_t$.

Moreover, using \textit{theorem 2.6} of [5] , we have
\begin{equation}
 \mathbb{P } \big ( \forall \, s \in [t,T] : u^\varepsilon (s, X_s^{t, \xi} ) = Y_s^{t, \xi} \big ) = 1
 \end{equation}
 \begin{equation}
 \mathbb{P} \otimes \mu \big (  \big \{ (\omega, s) \in \Omega \times [t,T]  : \mid Z_s^{t, \varepsilon, x } (\omega) \mid \geq \Gamma_1 \big  \}  \big )
 \end{equation}
 
\begin{flushleft}
 where $\mu$ stands for the Lebesgue measure in the real line and $\Gamma_1$ is a constant which only depends on $L, \Lambda, k, d,T$. 
\end{flushleft}
  By remark 2.7 of [5] we have that, for each $\varepsilon > 0$,  $u^\varepsilon$  only depends on the coefficients of the system (2.1), $f,g, \sqrt{\varepsilon}\sigma, h$. The fact that the dependence of $\Gamma_1 = \Gamma_1 (L, \Lambda, k,d,T)$ determines that the properties (3.1)-(3.2)  above hold uniformly in $\varepsilon$; in particular, there exist continuous versions and uniformly bounded (in $\varepsilon$ too) of $(Y_s^{t, \varepsilon, x},  Z_s^{t, \varepsilon, x} )_{t \leq s \leq T}$.
  
  We  now assume more regularity on the coefficients of (1.1) in the next set of assumptions.
\begin{flushleft}
 \textbf{Assumption B} 
\end{flushleft}

For $T \leq C$, we say that $f,g,h, \sigma$ satisfy $(A.B)$ if there exists three \\ constants $\lambda$, $\Lambda$ $\gamma >0 $

\vspace{-10 pt}

\begin{flushleft}
\textit{(B.1)}
  $\forall\, t \in [0,T]$, $\forall \, (x,y,z)\in \mathbb{R}^d \times \mathbb{R}^k \times \mathbb{R}^{k \times d}$:
\end{flushleft}
  
  \vspace{-20 pt} 
  
  \begin{flushleft}
$ \mid f(t,x,y) \mid \leq \Lambda (1 + \mid y \mid )$ \\
$\mid g(t,x,y,z) \mid \leq \Lambda (1 + \mid y \mid + \mid z \mid) $\\
$\mid \sigma (t,x,y) \mid \leq \Lambda$\\
$\mid h(x) \mid \leq \Lambda $ \\
  \end{flushleft}
  
  \vspace{1 pt}

\begin{flushleft}
\textit{(B.2)}
   $\forall \, (t,x,y)\in [0,T] \times \mathbb{R}^d \times \mathbb{R}^k \times \mathbb{R}^{k \times d}$:\end{flushleft}
   
   \vspace{-20 pt}
   
   \begin{flushleft}
   $<\xi, a(t,x,y) \xi> \geq \lambda  \mid \xi \mid^2 \,\,\, \forall \,\, \xi \in \mathbb{R}^d
 \,\,$where$\,\, a(t,x,y) = \sigma \sigma^T (t,x,y).$\\
   \end{flushleft}

\vspace{1 pt}

\begin{flushleft}
\textit{(B.3)} the function $\sigma$ is continuous on its definition set. \end{flushleft} 

\vspace{1 pt}

\begin{flushleft}
\textit{(B.4)} The function $\sigma$ is differentiable with respect to $x$ and $y$ and its derivatives with respect to $x$ and $y$ are $\gamma$- H\"{o}lder in $x$ and $y$, uniformly in $t$. 
\end{flushleft}

\vspace{1 pt}

Under the set of assumptions \textit{(A.A)} and \textit{(A.B)},  using \textit{proposition 2.4} and \textit{proposition B.6} of [5] and \textit{theorem 2.9} of [6], one can prove that there exists two constants, $\kappa, \kappa_1$,  only depending on $L, \Lambda, d,k,T$ (independent of $\varepsilon$) such that :
\begin{equation}
\mid u^\varepsilon (t,x) \mid \leq \kappa 
\end{equation}
\begin{equation}
u^\varepsilon \, \in C_b^{1,2} ([0,T] \times \mathbb{R}^d )
\end{equation}
\begin{equation}
\displaystyle \sup_{(t,x) \in [0,T] \times \mathbb{R}^d} \mid \nabla_x u^\varepsilon (t,x) \mid \leq \kappa_1 
\end{equation}
and

\vspace{-19 pt}

\begin{equation}
Z_s^{t, \varepsilon, x} = \sqrt{\varepsilon} \nabla_x u^\varepsilon (t, X_s^{t, \varepsilon, x}) \sigma (s, X_s^{t, \varepsilon}, Y_s^{t, \varepsilon, x} )
\end{equation}
where $u^\varepsilon$ solves uniquely (2.2) in $C_b^{1,2} ([0,T] \times \mathbb{R}^d)$, the space of $C^1$ functions with respect to the first variable and $C^2$ with respect to the second  variable, with bounded derivatives.\\
All these facts can be proven probabilistically. The last claim and the proper-\\ties (3.3), (3.4) and (3.6) are proved in [5]. Delarue delivers in the appendices of [5] probabilistic methods to obtain these regularity results, under assumptions (A.A) and (A.B) and over a small time enough duration, using \textit{Malliavin} Calculus techniques. The estimate of the gradient (3.5) is establi-\\shed by a probabilistic scheme in [6], using a variant of the \textit{Malliavin-Bismut} integration by parts formula proposed by Thalmaier [19] and applied in [20] to establish a gradient estimate of interior type for the solutions of a linear elliptic equation on a manifold. 

\begin{theorem} {\textbf{Assymptotic Behaviour when $\varepsilon \rightarrow 0$.}} 
\newline
Under the previous assumptions\textit{ (A)}, and \textit{ (B)}, assuming further $T \leq C$, where $C= C(L, \gamma)$ depending on the Lipschitz constant $L$ and on  a constant $\gamma$ that appears in the middle of the proof  (see 3.13) we have

\begin{flushleft}
1. For each $s,t  \in [0,T]$, $t \leq s$, the solution  of (2.1), $(X_s^{t, \varepsilon, x}, Y_s^{t, \varepsilon, x}, Z_s^{t, \varepsilon, x})\newline_{t \leq s \leq T}$ converges in $\mathcal{M}_t$, when $\varepsilon \rightarrow 0$ to $(X_s, Y_s, 0)_{t \leq s \leq T}$, where  $(X_s, Y_s)_{t \leq s \leq T}$ solves the  coupled system of differential equations:
\end{flushleft}

\begin{equation}
\begin{cases}
\dot{X}_s = f (s, X_s, Y_s) \\
\dot{Y}_s = -g (s, X_s, Y_s, 0), \,\, t \leq s \leq T\\ 
X_t = x ,
Y_T  = h(X_T) 
\end{cases}
\end{equation}

\begin{flushleft}
2.  Denoting $u(t,x) = Y_t^{t,x}$ , the limit in $\varepsilon \rightarrow 0$ of $Y_t^{t, \varepsilon,x}$, the function $u$ is a viscosity solution of
\end{flushleft}

\vspace{-30 pt}

\begin{equation}
\begin{cases}
\dfrac{\partial u^l}{\partial t}  + \displaystyle \sum_{i=1}^d f_i (t,x, u(t,x)) \dfrac{\partial u^l}{\partial x_i} (t,x) + 
 g^l (t,x, u(t,x), 0) = 0 \\
 u^\varepsilon (t,x) = h (x); \,\, x \in \, \mathbb{R}^d ; \,\, t \in [0,T] ; \,\,
 l=1,...,k
\end{cases}
\end{equation}

\vspace{-20 pt}

\begin{flushleft}
3. The function $u$ is bounded, continuous Lipschitz in $x$ and uniformly continuous in time.

\end{flushleft}

\begin{flushleft}
4. Furthermore,  if $u \in C_b ^{1,1} ([0,T] \times \mathbb{R}^d)$, since  that (3.7) has a unique continuous solution, due to the assumptions \textit{A} and \textit{B}, the function $u$ is a classical  solution of (3.8). 
\end{flushleft}

\end{theorem}

\begin{proof}
 Given $\varepsilon, \varepsilon_1 > 0$, $x \in \mathbb{R}^d$, for $T \leq C$, if  $(X_s^{t, \varepsilon, x}, Y_s^{t, \varepsilon, x}, Z_s^{t, \varepsilon, x})_{t \leq s \leq T}$ is the unique solution in $\mathcal{M}_t$ of
 
 \vspace{-10 pt}
 
 \begin{equation}
\begin{cases}
X_s^{t, \varepsilon, x} = x + \displaystyle \int_t^s f(r, X_r^{t, \varepsilon,x}, Y_r^{t, \varepsilon, x} ) \mathrm{d}r + \sqrt{\varepsilon} \displaystyle \int_t^s \sigma (r, X_r^{t, \varepsilon, x}, Y_r^{t, \varepsilon, x} ) \mathrm{d}Br \\
Y_s^{t, \varepsilon, x} =   h(X_T^{t,\varepsilon, x}) + \displaystyle \int_s^T g(r, X_r^{t, \varepsilon, x}, Y_r^{t, \varepsilon, x}, Z_r^{t, \varepsilon , x }) \mathrm{d}r -  \displaystyle \int_s^T Z_r^{t, \varepsilon, x }   \mathrm{d}Br \\
x \,\, \in \, \mathbb{R}^d; \,\,\, 0 \leq t \leq s \leq T \
\end {cases}
\end{equation}
 
\begin{flushleft}
and 
 $(X_s^{t, \varepsilon_1, x}, Y_s^{t, \varepsilon_1, x}, Z_s^{t, \varepsilon_1, x})_{t \leq s \leq T}$ is the unique solution in $\mathcal{M}_t$ of
\end{flushleft}
 
 \vspace{-5 pt}
 \begin{equation}
\begin{cases}
X_s^{t, \varepsilon_1, x} = x + \displaystyle \int_t^s f(r, X_r^{t, \varepsilon_1,x}, Y_r^{t, \varepsilon_1, x} ) \mathrm{d}r + \sqrt{\varepsilon_1} \displaystyle \int_t^s \sigma (r, X_r^{t, \varepsilon_1, x}, Y_r^{t, \varepsilon_1, x} ) \mathrm{d}Br \\
Y_s^{t, \varepsilon_1, x} =   h(X_T^{t, \varepsilon_1, x}) + \displaystyle \int_s^T g(r, X_r^{t, \varepsilon_1, x}, Y_r^{t, \varepsilon_1, x}, Z_r^{t, \varepsilon_1 , x }) \mathrm{d}r -  \displaystyle \int_s^T Z_r^{t, \varepsilon_1, x }   \mathrm{d}Br \\
x \,\, \in \, \mathbb{R}^d;  \,\,\, 0 \leq t \leq s \leq T \
\end {cases}
\end{equation}

\begin{flushleft}
\begin{flushleft}
Write, for sake of simplicity on the notations, with $\delta = \varepsilon, \varepsilon_1$ \\

\end{flushleft}
\end{flushleft}

\vspace{-10 pt}

\begin{flushleft}
$ f^\delta (s) = f(s, X_s^{t, \delta, x}, Y_s^{t, \delta, x})$ \\
 $g^\delta (s) = g (s, X_s^{t, \delta, x}, Y_s^{t, \delta, x}, Z_s^{t, \delta, x} )$\\
 $\sigma^\delta (s) = \sqrt{ \delta } \sigma (s, X_s^{t, \delta, x} , Y_s^{t, \delta, x})$ \\
 $h^\delta = h (X_T^{t, \delta, x})$\\ 
 \end{flushleft}

\vspace{-20 pt}

\begin{flushleft}
  As in the proof of \textit{theorem 1.2} of [5], using\textit{ It\^{o}'s Formula}:
\end{flushleft} 

 $\mathbb{E}   \displaystyle \sup_{t \leq s \leq T} \mid X^{\varepsilon}_s - X_s^{\varepsilon_1} \mid^2 \leq  \mathbb{E} \displaystyle \int_t^T   \mid \sigma^\varepsilon (s) - \sigma^{\varepsilon_1} (s) \mid^2 \mathrm{d}s$
 
  $+ 2 \mathbb{E}\, \displaystyle \sup_{t \leq s \leq T} \displaystyle \int_t^s \, <X^{\varepsilon}_r - X^{\varepsilon_1}_r, f^\varepsilon (r)- f^{\varepsilon_1} (r)> \mathrm{d}r  $
  
\vspace{-10 pt}  
  
  \begin{equation}
 + 2 \mathbb{E} \, \displaystyle \sup_{t \leq s \leq T} \displaystyle \int_t^s \,<X^{\varepsilon}_r - X^{\varepsilon_1}_r, (\sigma^\varepsilon (r) - \sigma^{\varepsilon_1} (r) ) \mathrm{d}Br> 
 \end{equation}
 
 \vspace{-1 pt}
Using \textit{Burkholder-Davis-Gundy's inequalities}, there exists $\gamma > 0$ such that:

 $\mathbb{E}   \displaystyle \sup_{t \leq s \leq T} \mid X^{\varepsilon}_s - X_s^{\varepsilon_1} \mid^2 \leq  \mathbb{E} \displaystyle \int_t^T   \mid \sigma^\varepsilon (s) - \sigma^{\varepsilon_1} (s)  \mid^2 \mathrm{d}s$

$+ 2\mathbb{E} \, \displaystyle \sup_{t \leq s \leq T} \displaystyle \int_t^s \,  <X_r^{\varepsilon} - X^{\varepsilon_1}_r, f^\varepsilon (r)- f^{\varepsilon_1} (r)> \mathrm{d}r$ 

 \vspace{-10 pt}

\begin{equation}
 + 2 \gamma  \mathbb{E}\, \Big ( \displaystyle \int_t^T \,\mid X^{\varepsilon}_r - X^{\varepsilon_1}_r \mid^2  \mid \sigma^\varepsilon (r)- \sigma^{\varepsilon_1} (r) \mid^2 \mathrm{d}r \Big ) ^{1/2}
\end{equation}

\begin{flushleft}
where $\gamma > 0$ only depends on the Lipschitz constant $ L$ and  $T$.
\end{flushleft}

\vspace{-5 pt}

\begin{flushleft}
Using the Lipschitz property (A.1)  , there exists $\gamma >0$,  eventually different, but only dependent on $L,T$ such that:
\end{flushleft}

$\mathbb{E} \displaystyle \sup_{t \leq s \leq T} \mid X^{\varepsilon}_s - X^{\varepsilon_1}_s \mid^2 
\leq \gamma 
\Big \{  \mathbb{E} \displaystyle \int_t^T  ( \mid X^{\varepsilon}_s  - X^{\varepsilon_1}_s \mid^2 +  \mid Y^{\varepsilon}_s  - Y^{\varepsilon_1}_s \mid^2 ) \mathrm{d}s  $

$+  \mathbb{E} \displaystyle \int_t^T   \mid \sigma^{\varepsilon} (s)  - \sigma^{\varepsilon_1} (s) \mid^2  \mathrm{d}s + 2 \mid \sqrt{\varepsilon} - \sqrt{\varepsilon_1} \mid^2  $

\vspace{-10 pt}

\begin{equation}
+  2 \sqrt{T} \displaystyle \sup_{t \leq s \leq T} \mid X^{\varepsilon}_s - X^{\varepsilon_1}_s \mid^2   \Big \}
\end{equation}

\begin{flushleft}
Then, assuming that $1- 2 \gamma \sqrt{T} \geq 0$
\end{flushleft}

$(1- 2 \gamma \sqrt{T} )\displaystyle \sup_{t \leq s \leq T} \mid X^{\varepsilon}_s - X^{\varepsilon_1}_s \mid^2 \\
 \leq \gamma \Big \{ \mathbb{E}   \displaystyle \int_t^T  ( \mid X^{\varepsilon}_s  - X^{\varepsilon_1}_s \mid^2 +  \mid Y^{\varepsilon}_s  - Y^{\varepsilon_1}_s \mid^2 ) \mathrm{d}s +  2 \mid \sqrt{\varepsilon} - \sqrt{\varepsilon_1} \mid^2     $

\vspace{-10 pt}

\begin{equation}
+ \mathbb{E}  \displaystyle \int_t^T   \mid \sigma^{\varepsilon} (s)  - \sigma^{\varepsilon_1} (s) \mid^2  \mathrm{d}s \Big \}
\end{equation}

\begin{flushleft}
 Using   estimates as before, we get for some new $\gamma \geq 0$ eventually:
\end{flushleft}

$\mathbb{E} \Big (  \displaystyle \sup_{t \leq s \leq T} \mid Y^{\varepsilon}_s - Y^{\varepsilon_1}_s \mid^2 \Big )  +  \mathbb{E}  \Big ( \displaystyle \int_t^T \mid Z^{\varepsilon}_s - Z^{\varepsilon_1}_s\mid^2 \mathrm{d}s \Big )$ 

$\leq \gamma  \Big [ \mathbb{E} \mid h^\varepsilon - h^{\varepsilon_1} \mid^2 + \mathbb{E} \Big ( \displaystyle \sup_{t \leq s \leq T} \displaystyle \int_t^T < Y^\varepsilon_s - Y^{\varepsilon_1}_s,  g^\varepsilon(s) - g^{\varepsilon_1}(s)> \mathrm{d}s \Big ) \Big ] $

\begin{flushleft}
Using  (3.13)  and the assumption (A.1)   modifying $\gamma$ eventually:
\end{flushleft}

$\mathbb{E}   \displaystyle \sup_{t \leq s \leq T} \mid X^{\varepsilon}_s - X^{\varepsilon_1}_s \mid^2 + \mathbb{E}   \displaystyle \sup_{t \leq s \leq T} \mid Y^{\varepsilon}_s - Y^{\varepsilon_1}_s \mid^2  +  \mathbb{E}   \displaystyle \int_t^T \mid Z^{\varepsilon}_s - Z^{\varepsilon_1}_s \mid^2 \mathrm{d}s $ 

$\leq \gamma \Big \{  \mathbb{E}  \mid h^{\varepsilon} - h^{\varepsilon_1} \mid^2 + \mathbb{E} \displaystyle \int_t^T  \mid X^{\varepsilon}_s - X^{\varepsilon_1}_s \mid^2 \mathrm{d}s  +  \mathbb{E} \displaystyle \int_t^T  \mid Y^{\varepsilon}_s - Y^{\varepsilon_1}_s \mid^2 \mathrm{d}s $

$+ \mathbb{E} \Big ( \displaystyle \int_t^T  (\mid f^\varepsilon (s) - f^{\varepsilon_1} (s) \mid + \mid g^{\varepsilon}(s) - g^{\varepsilon_1} (s)  \mid  )\mathrm{d}s \Big )^2 $

$+ \mathbb{E} \displaystyle \int_t^T \mid \sigma^{\varepsilon} (s) - \sigma^{\varepsilon_1} (s) \mid^2  \mathrm{d}s  + \mid \sqrt{\varepsilon} - \sqrt{\varepsilon_1} \mid^2 $

+ $\mathbb{E} \displaystyle \int_t^T  \mid Z^{\varepsilon}_s - Z^{\varepsilon_1}_s \mid ( \mid Y^{\varepsilon}_s - Y^{\varepsilon_1}_s\mid + \mid X^{\varepsilon}_s - X^{\varepsilon_1}_s\mid ) \mathrm{d}s
\Big \} $

\begin{flushleft}
So, there exist $C^* \leq C $ and $\Gamma$, only depending on $L$ and $\gamma_1$, such that, for $T-t \leq C^*$
\end{flushleft}

 $\mathbb{E}   \displaystyle \sup_{t \leq s \leq T} \mid X_s^\varepsilon - X_s^{\varepsilon_1} \mid^2 +  \mathbb{E} \displaystyle \sup_{t \leq s \leq T} \mid Y_s^\varepsilon - Y_s^{\varepsilon_1} \mid^2   + \mathbb{E} \displaystyle \int_t^T \mid Z_s^\varepsilon - Z_s^{\varepsilon_1} \mid^2  \mathrm{d}s $ 
 
 $ \leq \gamma \Big \{  \mathbb{E} \mid h^\varepsilon - h^{\varepsilon_1} \mid^2   + \mathbb{E} \displaystyle \int_t^T \mid  \sigma^\varepsilon (s) -  \sigma^{\varepsilon_1}  (s) \mid^2 \mathrm{d}s  + \mid \sqrt{\varepsilon} - \sqrt{\varepsilon_1} \mid^2 $
 
  \begin{equation}
    +   \mathbb{E} \big ( \displaystyle \int_t^T \mid Y_s^{\varepsilon} - Y_s^{\varepsilon_1} \mid + \mid g^\varepsilon (s) - g^{\varepsilon_1}  (s) \mid  \mathrm{d}s\big )^2   \Big \} 
  \end{equation}
  
  We can repeat the argument in $[T- 2 C^* ; T - C^*]$ and recurrently we get this important a-priori estimate (3.15) valid for  all $[t,T]$, with some (possibly different)  constant $\gamma >0$ only depending on $L$.
  
  Now, from (A.1), the following estimate holds:
  
  \begin{equation}
  \mathbb{E} \mid h^\varepsilon - h^{\varepsilon_1}  \mid^2  \leq L^2 \,\, \mathbb{E} \displaystyle \sup_{t \leq s \leq T} \mid X_s^\varepsilon - X_s^{\varepsilon_1} \mid^2
  \end{equation}
  
  One can also derive the estimate
  \begin{equation*}
   \mathbb{E}  \displaystyle \int_t^T \Big (\mid \sqrt{\varepsilon} \sigma^\varepsilon (s) - \sqrt{\varepsilon_1} \sigma^{\varepsilon_1}  (s)  \mid^2  \Big ) \mathrm{d}s  \leq c_1 \mid \sqrt{\varepsilon} - \sqrt{\varepsilon_1} \mid^2 
  \end{equation*}
  \begin{equation}
     +   c_2 \,\,\, \mathbb{E} \displaystyle \int_t^T \big ( \mid X_s^\varepsilon - X_s^{\varepsilon_1} \mid^2 + \mid Y_s^\varepsilon - Y_s^{\varepsilon_1} \mid^2  \big )\mathrm{d}s
  \end{equation}
  
 \begin{flushleft}
 where $c_1, c_2 > 0$ only depending  on $T$,$L$ and on the bound of $\mid\sigma\mid$.\end{flushleft}
 
   Furthermore, using \textit{Jensen's inequality} and the Lipschitz property of $g$, for some $c_3$ modified eventually along the various steps,
   
   $\mathbb{E} \big ( \displaystyle \int_t^T \mid Y_s^{\varepsilon} - Y_s^{\varepsilon_1} \mid + \mid g^\varepsilon (s) - g^{\varepsilon_1}  (s) \mid  \mathrm{d}s\big )^2 $ 
  
$\leq  \mathbb{E}  \displaystyle \int_t^T \mid Y_s^\varepsilon - Y_s^{\varepsilon_1} \mid^2 + \mid g^\varepsilon - g^{ \varepsilon_1}\mid^2 + 2 \mid   Y_s^\varepsilon - Y_s^{\varepsilon_1} \mid \mid g^\varepsilon - g^{ \varepsilon_1} \mid \mathrm{d}s $
 
  $\leq c_3 \,\, \mathbb{E} \Big [ \displaystyle \int_t^T \mid X_s^\varepsilon - X_s^{\varepsilon_1}\mid^2 +\mid Y_s^\varepsilon - Y_s^{\varepsilon_1}\mid^2 + \mid Z_s^\varepsilon - Z_s^{\varepsilon_1}\mid^2 \Big ] $

  \vspace{-15 pt}
  
  \begin{equation}
  \leq  c_3 \,\, \Big \{ \mathbb{E} \displaystyle \sup_{t \leq s \leq T} \mid X_s^\varepsilon - X_s^{\varepsilon_1} \mid^2 + \mathbb{E} \displaystyle \sup_{t \leq s \leq T} \mid Y_s^\varepsilon - Y_s^{\varepsilon_1} \mid^2 + \mathbb{E} \displaystyle \sup_{t \leq s \leq T} \mid Z_s^\varepsilon - Z_s^{\varepsilon_1} \mid^2 \Big \}
  \end{equation}

\begin{flushleft}
Using (3.6) and the boundedness of $\mid  \sigma \mid$, we have
\end{flushleft}

$\mathbb{E}  \displaystyle \int_t^T \mid Z_s^\varepsilon \mid^2 \mathrm{d}s  \rightarrow 0 $  as  $ \varepsilon \rightarrow 0 $  and

\vspace{-10 pt}

\begin{equation}
\displaystyle \lim_{ \mid \varepsilon - \varepsilon_1 \mid \rightarrow 0}  \mathbb{E}  \displaystyle \int_t^T \mid Z_s^\varepsilon  - Z_s^{\varepsilon_1} \mid^2 \mathrm{d}s = 0
\end{equation} 

 Furthermore, by \textit{Burkholder-Davis-Gundy's inequalities}, there exists a constant $\gamma > 0$, eventually new, such that: 
 
 $\mathbb{E} \displaystyle \sup_{t \leq s \leq T} \mid X_s^\varepsilon - X_s^{\varepsilon_1}\mid^2 +\mathbb{E} \displaystyle \sup_{t \leq s \leq T} \mid Y_s^\varepsilon - Y_s^{\varepsilon_1}\mid^2 + \mathbb{E} \displaystyle \int_t^T \mid Z_s^\varepsilon - Z_s^{\varepsilon_1}\mid^2 \mathrm{d}s $

 $\leq \gamma \mathbb{E} \displaystyle \int_t^T \mid X_s^\varepsilon - X_s^{\varepsilon_1} \mid^2 + \mid Y_s^\varepsilon - Y_s^{\varepsilon_1} \mid^2 + \mid  Z_s^\varepsilon - Z_s^{\varepsilon_1} \mid^2 \mathrm{d}s$

 \vspace{-10 pt}
 
 \begin{equation}
 \leq \gamma \mathbb{E} \displaystyle \int_t^T \displaystyle \sup_{t \leq r \leq s} \mid X_r^\varepsilon - X_r^{\varepsilon_1} \mid^2 + \displaystyle \sup_{t \leq r \leq s} \mid Y_r^\varepsilon - Y_r^{\varepsilon_1} \mid^2 + \displaystyle \sup_{t \leq r \leq s} \mid  Z_r^\varepsilon - Z_r^{\varepsilon_1} \mid^2 \mathrm{d}s  
 \end{equation}
 
 Moreover, for some $\gamma_1, \gamma_2 > 0$ , using (3.18)- (3.21)
 
 $\mathbb{E} \displaystyle \sup_{t \leq s \leq T} \mid X_s^\varepsilon - X_s^{\varepsilon_1}\mid^2 +\mathbb{E} \displaystyle \sup_{t \leq s \leq T} \mid Y_s^\varepsilon - Y_s^{\varepsilon_1}\mid^2 + \mathbb{E} \displaystyle \int_t^T \mid Z_s^\varepsilon - z_s^{\varepsilon_1}\mid^2 \mathrm{d}s$ 
 
$\leq  \gamma_1 \mid \sqrt{\varepsilon} - \sqrt{\varepsilon_1} \mid^2 + 
  \gamma_2 \,\, \mathbb{E} \displaystyle \int_t^T \displaystyle \sup_{t \leq r \leq s} \mid X_r^\varepsilon - X_r^{\varepsilon_1} \mid^2 + \displaystyle \sup_{t \leq r \leq s} \mid Y_r^\varepsilon - Y_r^{\varepsilon_1} \mid^2$ 
  
  \vspace{-10 pt}
  
  \begin{equation}
   +  \displaystyle \sup_{t \leq r \leq s} \mid  Z_r^\varepsilon - Z_r^{\varepsilon_1} \mid^2 \mathrm{d}s
  \end{equation}
 
  Using \textit{Gronwall's inequality}, \\
  
   $\mathbb{E} \displaystyle \int_t^T \displaystyle \sup_{t \leq r \leq s} \mid X_r^\varepsilon - X_r^{\varepsilon_1} \mid^2 + \displaystyle \sup_{t \leq r \leq s} \mid Y_s^\varepsilon - Y_s^{\varepsilon_1} \mid^2 \leq$ 

 \begin{equation}
  \leq C_1  \mid \sqrt{\varepsilon} - \sqrt{\varepsilon_1} \mid^2  \mathbb{E} \displaystyle \sup_{t \leq s \leq T} \mid  Z_s^\varepsilon - Z_s^{\varepsilon_1} \mid^2 \rightarrow 0    
 $  as  $ \mid  \varepsilon - \varepsilon_1 \mid \rightarrow 0 .
 \end{equation}  
 for some $C_1 > 0$ only depending on $L$ (independent of $\varepsilon$), once  \\
 $ \mathbb{E} \displaystyle \sup_{t \leq s \leq T} \mid  Z_s^\varepsilon - Z_s^{\varepsilon_1} \mid^2$ is bounded,
 by the results (3.5) and (3.6)  .
 
We conclude, by the previous estimates,  that the pair $(X_s^{t, \varepsilon}, Y_s^{t, \varepsilon})_{t \leq s \leq T}$ form a Cauchy sequence and therefore  converges in $S^2_t(\mathbb{R}^d) \times S_t^2 (\mathbb{R}^k)$ . Denote $(X_s, Y_s)_{t \leq s \leq T }$ its limit. $(X_s, Y_s, 0)_{t \leq s \leq T}$ is the limit in $\mathcal{M}_t$ of   $(X_s^{t, \varepsilon}, Y_s^{t, \varepsilon}, Z_s^{t, \varepsilon, x})_{t \leq s \leq T}$ when $\varepsilon \rightarrow 0$.

If we consider the forward equation in (2.1) and if we take the limit pointwise when $\varepsilon \rightarrow 0$, we have

\vspace{-10 pt}

\begin{equation}
X_s^t = x + \displaystyle \int_t^s f(r,X_r, Y_r) \mathrm{d}r $   a.s  $ 
\end{equation}
 where we have used the boundedness of $\sigma$ and the continuity of $f$. 
 
Similarly, we can take  the limit on the backward equation when $\varepsilon \rightarrow 0$. Using the continuity of the functions $h,g$ and  

 $\mathbb{E} \big ( \displaystyle \int_s^T Z_r^{t, \varepsilon} \mathrm{d}Br \big )^2 = \mathbb{E} \displaystyle \int_s^T \mid Z_r^{t, \varepsilon}  \mid^2 \mathrm{d}r \rightarrow 0 $ as $\varepsilon \rightarrow 0$

  which implies $  \displaystyle \int_s^T Z_r^{t, \varepsilon} \mathrm{d}Br \rightarrow 0$, $a.s \,\, \mathbb{P}$, we have 
 \begin{equation}
Y_s^t = h(X_T) + \displaystyle \int_s^T g(r,X_r, Y_r, 0) \mathrm{d}r $  a.s  $ 
\end{equation}

 In conclusion, $(X_s, Y_s)_{t \leq s \leq T}$  solves the following deterministic problem of ordinary (coupled) differential equations, almost surely,
 
\begin{equation}
\begin{cases}
\dot{X}_s = f (s, X_s, Y_s) \\
\dot{Y}_s = -g (s, X_s, Y_s, 0), \,\, t \leq s \leq T\\
X_t = x ,
Y_T  = h(X_T)
\end{cases}
\end{equation}
 
 Given $t, t' \in [0,T]$, $x, x' \in \mathbb{R}^d$, consider $(X_s^{t, \varepsilon,x}, Y_s^{t, \varepsilon,x}, Z_s^{t, \varepsilon,x})_{t \leq s \leq T}$ the unique solution in $\mathcal{M}_t$
 
\begin{equation}
\begin{cases}
X_s^{t, \varepsilon, x} = x + \displaystyle \int_t^s f(r, X_r^{t, \varepsilon,x}, Y_r^{t, \varepsilon, x} ) \mathrm{d}r + \sqrt{\varepsilon} \displaystyle \int_t^s \sigma (r, X_r^{t, \varepsilon, x}, Y_r^{t, \varepsilon} ) \mathrm{d}Br \\
Y_s^{t, \varepsilon, x} =   h(X_T^{t,x, \varepsilon}) + \displaystyle \int_s^T g(r, X_r^{t, \varepsilon, x}, Y_r^{t, \varepsilon, x}, Z_r^{t, \varepsilon , x }) \mathrm{d}r -  \displaystyle \int_s^T Z_r^{t, \varepsilon, x }   \mathrm{d}Br \\
  t \leq s \leq T \
\end {cases}
\end{equation}

\vspace{-10 pt}

\begin{flushleft}
extended to the whole interval $[0,T]$, putting \end{flushleft}

\vspace{-20 pt}

\begin{equation}
\forall \,\, 0 \leq s \leq t \,\,\,\, X_s^{t, \varepsilon, x} =x; \,\, Y_s^{t, \varepsilon,x} =Y_t^{t, \varepsilon,x}; \,\, Z_s^{t, \varepsilon, x} = 0 
\end{equation}

\vspace{-20 pt}

\begin{flushleft}
and consider  $(X_s^{t', \varepsilon,x'}, Y_s^{t', \varepsilon,x'}, Z_s^{t', \varepsilon,x'})_{t' \leq s \leq T}$ the unique solution in $\mathcal{M}_{t'}$ \end{flushleft}

 \vspace{-15 pt}
 
\begin{equation}
\begin{cases}
X_s^{t', \varepsilon, x'} = x' + \displaystyle \int_{t'}^s f(r, X_r^{t', \varepsilon,x'}, Y_r^{t', \varepsilon, x'} ) \mathrm{d}r + \sqrt{\varepsilon} \displaystyle \int_{t'}^s \sigma (r, X_r^{t', \varepsilon, x'}, Y_r^{t', \varepsilon, x'} ) \mathrm{d}Br \\
Y_s^{t', \varepsilon, x'} =   h(X_T^{t', \varepsilon, x'}) + \displaystyle \int_s^T g(r, X_r^{t', \varepsilon, x'}, Y_r^{t', \varepsilon, x'}, Z_r^{t', \varepsilon , x' }) \mathrm{d}r -  \displaystyle \int_s^T Z_r^{t', \varepsilon, x' }   \mathrm{d}Br \\
  t' \leq s \leq T \
\end {cases}
\end{equation}

\begin{flushleft}
extended to the whole interval $[0,T]$, putting
\end{flushleft} 
\begin{equation}
\forall \,\, 0 \leq s \leq t' \,\,\,\,  X_s^{t', \varepsilon, x'} =x; \,\, Y_s^{t', \varepsilon,x'} =Y_{t'}^{t', \varepsilon,x'}; \,\, Z_s^{t', \varepsilon, x'} = 0 
\end{equation}

Using the estimate (3.15), such as in the \textit{corollary 1.4} of [5], we are lead to

$\mathbb{E} \displaystyle \sup_{0 \leq s \leq T} \mid X_s^{t, \varepsilon,x} - X_s^{t', \varepsilon, x'} \mid^2 + \mathbb{E} \displaystyle \sup_{0 \leq s \leq T} \mid Y_s^{t, \varepsilon,x} - Y_s^{t', \varepsilon, x'} \mid^2 $

\begin{equation}
+\mathbb{E} \displaystyle \int_0^T \mid Z_s^{t, \varepsilon,x} - Z_s^{t', \varepsilon, x'} \mid^2 \leq  \alpha \mid x - x' \mid^2 + \beta (1 + \mid x \mid^2) \mid t - t' \mid^2
\end{equation}

\begin{flushleft}
where $\alpha, \beta > 0$ are constants only depending on $L, \Lambda$.
\end{flushleft}

Finally, 

\vspace{-15 pt}

\begin{equation}
\mid u^\varepsilon (t,x) - u^\varepsilon (t',x') \mid^2 \leq \alpha \mid x - x' \mid^2 + \beta (1 + \mid x \mid^2) \mid t - t' \mid^2
\end{equation}

\begin{flushleft}

which proves that $(u^\varepsilon)_{\varepsilon >0}$ is a family of equicontinuous maps on every compact set of $[0,T] \times \mathbb{R}^d$. We can apply \textit{Arzela's Ascoli Theorem} and conclude that the  convergence of $u^\varepsilon$ to $u$, where $u(t,x) = Y_t^{t,x}$, is uniform in $[0,T] \times K$, for every $K$ compact subset of $\mathbb{R}^d$.  \end{flushleft} 

Taking the limit in $\varepsilon \rightarrow 0$ in (3.31) we get 

\vspace{-20 pt}

\begin{equation}
\mid u (t', x') - u(t,x) \mid^2 \leq  \alpha \mid x - x' \mid^2 + \beta (1 + \mid x \mid^2) \mid t - t' \mid^2
\end{equation}

\begin{flushleft}
for all $\ (t,x) , (t', x') \in [0,T] \times \mathbb{R}^d$, which proves that the function $u$, limit in $\varepsilon \rightarrow 0$, is Lipschitz continuous in $x$ and uniformly continuous in $t$. The boundedness of $u$ is given by (3.3), since this bound is uniform in $\varepsilon$. \end{flushleft}

\vspace{20 pt}

 Using \textit{theorem 5.1} of [15] we deduce  clearly that $u^\varepsilon$ is a viscosity solution in $[0,T] \times \mathbb{R}^d$ of (2.2). 
 
 Since the coefficients of the quasilinear parabolic system are Lipschitz continuous, by the property of the compact uniform convergence of viscosity solutions for quasilinear parabolic equations (see [4] for details), we conclude that $u$ is a viscosity solution of (3.8).
\newline 
 
Moreover, let $v : [0,T] \times \mathbb{R}^d \rightarrow \mathbb{R}^k$ be a  $C_b^{1,1} ([0,T], \mathbb{R}^k) $ solution, continuous Lipschitz in $x$ and uniformly continuous in $t$ for (3.8). Fixing $(t,x) \in [0,T] \times \mathbb{R}^d$, we take the following function:

$\psi : [t, T] \rightarrow \mathbb{R}^k $ 

\vspace{5 pt}

  $\psi (s) := v (s, X_s^{t, x})$.
  \vspace{10 pt}

 Computing its time derivative:
\vspace{10 pt}

$\dfrac{d \psi}{d s} (s) $

$ = \dfrac{\partial v}{\partial s} (s, X_s^{t,x}) + \displaystyle \sum_{i=1}^d \dfrac{\partial v}{\partial x_i} (s, X_s^{t,x}) \dfrac{\partial (X_s^{t,x})}{\partial t}$ 
 
$= \dfrac{\partial v}{\partial s} (s, X_s^{t,x}) + \displaystyle \sum_{i=1}^d \dfrac{\partial v}{\partial x_i} (s, X_s^{t,x}) f(s, X_s^{t,x}, Y_s^{t,x})$ 

 $= - $g$ (s, X_s^{t,x}, v(s, X_s^{t,x}), 0)$  \\

$\psi (T) = v (T, X_T^{t,x}) = h(x)$ \\

As a consequence,  $v (t,x) = v (t, X_t^{t,x}) = u(t,x)$,  under the hypothesis of uniqueness of solution for the system of ordinary  differential equations (3.7). 
So, under these hypothesis, we have a uniqueness property of solution for (3.8) in the class of $C_b^{1,1} ([0,T] \times \mathbb{R}^d) $, which are Lipschitz continuous in $x$ and uniformly continuous in $t$.   

\end{proof}

\section{A Large Deviations Principle}

In this section we state a Large Deviation Principle  for the laws of the processes $(X_s^{t, \varepsilon, x}, Y_s^{t, \varepsilon, x})_{t \leq s \leq T}$ when the parameter $\varepsilon \rightarrow 0$. 

The main property employed to establish the Large Deviation Principle for $(X_s^{t, \varepsilon, x}, Y_s^{t, \varepsilon, x})_{t \leq s \leq T}$ is 

\vspace{-20 pt}

\begin{equation}
Y_s^{t, \varepsilon, x} = u^\varepsilon (s, X_s^{t, \varepsilon, x})  $ a.s$
\end{equation}

 That property  is the key to generalize the Large Deviation Principle  proved in [18] where  the FBSDE  is decoupled, namely  when $f(t,x,y) = f(t,x)$ and $\sigma (t,x,y) = \sigma (t,x)$.

For our purposes, we need the following results of  Large Deviations Theory.

\begin{theorem} {\textbf{Freidlin-Wentzell Estimates}} ([1],[2])
\newline
Given $t \in [0,T]$ and $x \in \mathbb{R}^d$, consider $b: [t,T] \times \mathbb{R}^d \rightarrow \mathbb{R}^d$ and $\sigma :  [t,T] \times \mathbb{R}^d \rightarrow \mathbb{R}^{d \times d }$ Lipschitz continuous functions, with sublinear growth and bounded in time. For each $\varepsilon > 0$, let  $ b^\varepsilon :  [t,T] \times \mathbb{R}^d \rightarrow \mathbb{R}^d$ and $\sigma^\varepsilon : [t,T] \times \mathbb{R}^d \rightarrow \mathbb{R}^{d \times d } $ be  Lipschitz continuous and with sublinear growth, such that:
\begin{equation}
\displaystyle \lim_{\varepsilon \rightarrow 0} \mid b^\varepsilon - b \mid = \displaystyle \lim_{\varepsilon \rightarrow 0} \mid \sigma^\varepsilon - \sigma \mid  = 0
\end{equation}
uniformly in the compact sets of $[t, T] \times \mathbb{R}^d$.

 The process $(X_s^\varepsilon)_{t \leq s \leq  T}$, the unique strong solution of the stochastic differential equation:
\begin{equation}
\begin{cases}
dX_s^\varepsilon = b^\varepsilon (s, X_s^{\varepsilon}) \mathrm{d}s + \sqrt{\varepsilon} \sigma^\varepsilon (s, X_s^\varepsilon) \mathrm{d}B_s \\
X_t^\varepsilon = x
\end{cases}
\end{equation}
satisfies a Large Deviations Principle in $C_x ([t,T], \mathbb{R}^d)$, the space of the conti- \\ nuous functions $\phi : [t,T] \rightarrow \mathbb{R}^d$ having  origin in $x$ with the good rate function:
\begin{equation}
I(g) = \displaystyle \inf \Big \{ \dfrac{1}{2} || \varphi ||^2_{H^1} : g \in H^1 ([t,T], \mathbb{R}^d) \,\,\, g_s = x + \displaystyle \int_t^s b(r, g_r) \mathrm{d}r + \displaystyle \int_t^s \sigma (r, g_r) \dot{\varphi}_r  \mathrm{d}r \Big \}
\end{equation}
This means that the level sets of I are compact and
\begin{equation}
\displaystyle \limsup_{\varepsilon \rightarrow 0} \varepsilon \log \mathbb{P} (X^\varepsilon \in F) \leq - \displaystyle \inf_{\psi \in F} I(\psi)
\end{equation}
\begin{equation}
\displaystyle \liminf_{\varepsilon \rightarrow 0} \varepsilon \log \mathbb{P} (X^\varepsilon \in G) \geq - \displaystyle \inf_{\psi \in G} I(\psi)
\end{equation}
for every closed set $F \in C_x ([t,T], \mathbb{R}^d)$ and for every open set $G \in C_x ([t,T], \mathbb{R}^d) $
\end{theorem} 

The next result is a very important tool in Large Deviations Theory, to transfer Large Deviations Principles from one topological space to another. We present a simple version of the Contraction Principle. 

\begin{theorem} {\textbf{Contraction Principle}  ([7]) 
\newline
If $f : X \rightarrow Y$ is a continuous mapping from a topological vector space $X$ to a metric space $(Y,d)$ and $\{\mu_\varepsilon\}_{\varepsilon > 0}$ satisfies a Large Deviations Principle with a good rate function 
$I : X \rightarrow [0, \infty]$, and for each $\varepsilon >0$ if $f_\varepsilon : X \rightarrow Y$ is a family of  continuous functions, uniformly convergent to $f$ in all  compact sets of $X$, we have that 

$\{ \mu_\varepsilon \circ f_\varepsilon^{-1}\}_{\varepsilon > 0}$ satisfies a Large Deviation Principle with the good rate function:

$J(y) = \displaystyle \inf \{ I(x) : x \in X $ and $  f(x) = y \}$}

\end{theorem}
\begin{flushleft}
\textbf{Remark}: We define the infimum over the empty set to be $+ \infty$. \\
\end{flushleft}
Under conditions (A.A) and (A.B) are in force, for $T \leq C$, we have 

\begin{theorem} {\textbf{A Large Deviations Principle}} 
\newline
When $\varepsilon \rightarrow 0 $, $ (X_s^{t, \varepsilon})_{t \leq s \leq T} $ obey a Large Deviation Principle in $C([t,T], \mathbb{R}^d)$ with the good rate function 

$I(g) = \displaystyle \inf \Big \{  \frac{1}{2} \displaystyle  \int_t^T \mid \dot{\varphi}_s\mid^2 \mathrm{d}s: g \in H^1([t,T], \mathbb{R}^d) , $

\begin{equation}
 g_s = x + \displaystyle \int_t^s f(r,g_r, u (r, g_r) ) \mathrm{d}r + \displaystyle \int_t^s \sigma (r, g_r, u(r, g_r)) \dot{\varphi}_r \mathrm{d}r  \,\, s \in [t,T]   \Big \}
 \end{equation}
  
for $\varphi \in C([t,T], \mathbb{R}^d)$ 

and $ (Y_s^{t, \varepsilon})_{t \leq s \leq T} $ obey a LDP in $C([t,T], \mathbb{R}^k)$ with the good rate function
\begin{equation}
J (\psi) = \displaystyle \inf \Big \{  I (\varphi ) :  F(\varphi) = \psi $ if  $ \psi \in H^1([t,T], \mathbb{R}^d )  \Big \}
\end{equation}

where $F(\varphi) (s) = u(s, \varphi_s)$  for all $\varphi \in C([t,T], \mathbb{R}^d)$ 

\end{theorem}

\begin{proof} 

Since $Y_s^{t, \varepsilon, x} = u^\varepsilon (s, X_s^{t, \varepsilon, x})$, the first equation on the FBSDE (2.1) is in the differential form given by (we omit the indices )

\vspace{-20 pt}

\begin{equation}
\begin{cases}
dX_s^\varepsilon = b^\varepsilon (s, X_s^\varepsilon) \mathrm{d}s + \sqrt{\varepsilon} \sigma_1^\varepsilon (s, X_s^\varepsilon) \mathrm{d}Bs; \,\, t \leq s \leq T \\
X_t^\varepsilon = x
\end{cases}
\end{equation}

if we write $b^\varepsilon (s, X_s^{\varepsilon})= f (s, X_s^{\varepsilon}, u^\varepsilon (s, X_s^{\varepsilon}))$ and $\sigma_1^{\varepsilon} (s, X_s^{\varepsilon})= \sigma (s, X_s^{\varepsilon}, u^\varepsilon (s, X_s^{\varepsilon})) $. 

We have  therefore a typical setting in which we can apply \textit{theorem 4.1}. 

Define:
\newline

 $b^\varepsilon : [0,T] \times \mathbb{R}^d  \rightarrow \mathbb{R}^d$

$b^\varepsilon (t,x) = f(t,x,u^\varepsilon (t,x))$ 

$\sigma_1^\varepsilon : [0,T] \times \mathbb{R}^d  \rightarrow \mathbb{R}^{d \times d }$

$\sigma_1^\varepsilon (t,x) = \sigma (t,x,u^\varepsilon (t,x))$ 
\newline

 We notice that $b^\varepsilon$  and $\sigma_1^\varepsilon$  are clearly Lipschitz continuous, with sub-\\linear growth. 
 
Since $\displaystyle \lim_{\varepsilon \rightarrow 0} \mid u^\varepsilon (t,x) - u(t,x) \mid = 0 $ uniformly in all compact sets of $[0,T] \times \mathbb{R}^d $ as we remarked in the\textit{ theorem 3.1} before,  by the Lipschitz pro-\\perty of the coefficients of the FBSDE (2.1)- assumption \textit{(A.1)}

\vspace{-10 pt}

\begin{equation}
 \displaystyle \lim_{\varepsilon \rightarrow 0} \mid b^\varepsilon (t,x) - b(t,x) \mid = \displaystyle \lim_{\varepsilon \rightarrow 0} \mid \sigma^\varepsilon (t,x) - \sigma (t,x) \mid = 0.  \\
\end{equation} 

\vspace{-10 pt}

\begin{flushleft}
uniformly in all compact sets of $[0,T] \times \mathbb{R}^d$, 

\end{flushleft}
where $b$ and $\sigma$ are defined by the obvious way:

$b(t,x) = f(t,x, u(t,x))$ and $\sigma (t,x) = \sigma (t,x, u(t,x))$. 

\vspace{-10 pt}

 \begin{flushleft}
By the standard  existence and uniqueness theory for ordinary differential equations, since $b$ and $\sigma_1$ are Lipschitz continuous,  the problem
\end{flushleft}
 
 \vspace{-15 pt}
 
 \begin{equation}
 \begin{cases}
 \dot{g}_s = b(s, g_s) \mathrm{d}s + \sigma (s, g_s) \dot{h}_s ; \,\,\, h \in H^1 ([t,T], \mathbb{R}^d) \\
 g_t = x
 \end{cases}
 \end{equation}
 
\vspace{-20 pt} 
 
\begin{flushleft}
has a unique solution in $C ( [t,T], \mathbb{R}^d)$;  we will denote it $g = S_x (h)$.\\
\end{flushleft}

We have defined an operator 
$S_x : H^1 ([t,T], \mathbb{R}^d) \rightarrow C ([t, T], \mathbb{R}^d)$ 
which is uniformly continuous (i.e. continuous for the supremum norm ). We can therefore  apply  \textit{theorem 4.1} and state that $(\mathbb{P} \circ (X^\varepsilon)^{-1})_{\varepsilon > 0}$ obey a Large Deviation Principle  with the good rate function 
\newline

$I(g) = \displaystyle \inf \Big \{  \frac{1}{2} \displaystyle  \int_t^T \mid \dot{\varphi}_s\mid^2 \mathrm{d}s: g \in H^1([t,T], \mathbb{R}^d) ,$ 

\vspace{-15 pt}

\begin{equation}
 g_s = x + \displaystyle \int_t^s f(r,g_r, u (r, g_r) ) \mathrm{d}r + \displaystyle \int_t^s \sigma (r, g_r, u(r, g_r)) \dot{\varphi}_r \mathrm{d}r  \,\, s \in [t,T]   \Big \}
 \end{equation}
  
  \vspace{-10 pt}
  
for $\varphi \in C([t,T], \mathbb{R}^d)$ 
\newline
 
In what follows, in order to prove a LDP to $(Y_s^\varepsilon)_{t \leq s \leq T}$ we consider the following operator: 
\newline

$F^\varepsilon : C( [t, T], \mathbb{R}^d) \rightarrow C( [t,T], \mathbb{R}^k )$ 

\vspace{-20 pt}

\begin{equation}
F^\varepsilon (\varphi ) (s) = u^\varepsilon (s, \varphi_s)
\end{equation}

 We observe that $Y_s^{\varepsilon} = F^\varepsilon (X_s^\varepsilon)$ for all $s \in [t, T]$. 
 
To establish a Large Deviation Principle  for $(\mathbb{P} \circ (Y^\varepsilon)^{-1})_{\varepsilon > 0}$ we will proceed using   the Contraction Principle in the form presented in \textit{theorem 4.2}.  

Proving  \textit{the continuity of $F^\varepsilon$: }

Let $\varepsilon > 0$ and $x \in C([t, T], \mathbb{R}^d)$. Let $(x_n)_{n \in \mathbb{N}}$ be a sequence in $C( [t,T], \mathbb{R}^d)$ converging to $x$ in the uniform norm. Fix $\delta > 0$. Since $\parallel x_n - x \parallel_{\infty} \rightarrow 0$, there exists $M > 0$ such that $\parallel x \parallel_{\infty}$ , $\parallel x_n \parallel_{\infty}  \leq M$. 

Due to \textit{theorem 4.1}, we know that  $u^\varepsilon $ is  a continuous function in $[0,T] \times \mathbb{R}^d$  and $u^\varepsilon$ is uniformly continuous in $[t, T] \times K $ where $K = \overline{B(0,M)} \subset \mathbb{R}^d$. 

There exists $\eta > 0$ such that for $s, s_1 \in [t,T]$ and $z, z_1 \in K$ $\mid s - s_1 \mid < \eta$ and $\mid z - z_1 \mid < \eta $, we have $\mid  u^\varepsilon (s,z) - u^\varepsilon (s_1, z_1) \mid < \delta$. 

Since $x_n \rightarrow x$ in $C([t,T], \mathbb{R}^d)$, fix $n_0 \in \mathbb{N}$ such that for all $n \geq n_0 $ we have $\parallel x_n - x\parallel_{\infty} < \eta$. 

For all $r \in [t, T]$ and for all $n \geq n_0$ $x_n (r), x(r) \in K$ and
\vspace{10 pt}

 $ \mid  u^\varepsilon (r,x(r) ) - u^\varepsilon (r, x_n (r) ) \mid < \delta$. 

\vspace{10 pt}

So, we conclude that $F^\varepsilon (x_n) \rightarrow F^\varepsilon (x)$ , which proves the continuity of $F^\varepsilon$ in the point $x \in C([t,T], \mathbb{R}^d)$. 

The next step consists in showing the  uniform convergence, in the compact sets of $C([t,T], \mathbb{R}^d)$ of $F^\varepsilon$  to $F$,  when $\varepsilon \rightarrow 0$, where $F(\varphi) (s) = u(s, \varphi_s)$  for all $\varphi \in C([t,T], \mathbb{R}^d)$.

Due  to the point 1 of the \textit{theorem 3.1},  we know that for all $x \in \mathbb{R}^d $

\vspace{-18 pt}

\begin{equation}
\mathbb{E} \displaystyle \sup_{t \leq s \leq T} \mid Y_s^{\varepsilon, t, x} - Y_s^{t, x} \mid^2 + \mathbb{E} \displaystyle \int_t^T \mid Z_s^{\varepsilon, t, x} \mid^2 \mathrm{d}s \rightarrow 0.
\end{equation}

Consider $K$ a compact set of $C( [t,T], \mathbb{R}^d)$ and let

$A: = \{ \varphi_s : \varphi \in K, s \in [t,T] \}$. 

 It is clear that $A$ is a compact set of $\mathbb{R}^d$. Using  (4.14)  we observe that 
 \newline

$\displaystyle \sup_{\varphi \in K} \parallel F^\varepsilon (\varphi) - F(\varphi) \parallel_{\infty}^2 = \displaystyle \sup_{\varphi \in K} \displaystyle \sup_{s \in [t, T]} \mid u^\varepsilon (s, \varphi_s) - u (s, \varphi_s)\mid^2 $

$= \displaystyle \sup_{\varphi \in K} \displaystyle \sup_{s \in [t, T]} \mid Y_s^{\varepsilon, s , \varphi_s} - Y_s^{ s, \varphi_s}\mid^2 $ 
 
 \vspace{-15 pt}
 
 \begin{equation}
 \leq \displaystyle \sup_{x \in A} \displaystyle \sup_{s \in [t, T]} \mid Y_s^{\varepsilon, s , x} - Y_s^{ s, x}\mid^2 \rightarrow 0 $  as   $ \varepsilon \rightarrow 0 
 \end{equation}

\begin{flushleft}
using the uniform convergence of $u^\varepsilon$ to $u$ in the compact sets of $[0,T] \times \mathbb{R}^d$.

\end{flushleft}
Using \textit{ theorem 4.2, } we conclude that $(\mathbb{P} \circ (Y^\varepsilon)^{-1})_{\varepsilon > 0}$ satisfies a LDP principle with the good rate function

\vspace{-20 pt}

\begin{equation}
J (\psi) = \displaystyle \inf \Big \{  I (\varphi ) :  F(\varphi) = \psi; \varphi \in H^1([t,T], \mathbb{R}^d)   \Big \}
\end{equation} 
for $\psi \in C([t, T], \mathbb{R}^k)$

\end{proof}

\section{Remarks and conclusions}

The results presented in this work - \textit{theorem 3.1} and \textit{theorem 4.3}, were stated under the assumptions \textit{(A.A)} and \textit{(A.B)}, which ensure existence and uniqueness of solution of (2.1) for a local time $T \leq C$, where $C$ is a certain constant only depending on the Lipschitz constant $L$. Under these assumption, our results and the existence and uniqueness of solution for the FBSDE (2.1) ([5], [14]) do not depend on results for PDE's but only on probabilistic arguments. However, it is possible to extend  theorem 3.1 and  theorem 4.3 to global time. In order to do it and to prove the important properties \textit{(3.3) - (3.5)} in [5],  results of deterministic quasilinear parabolic partial differential equations [11] are used. If we maintain the assumption of smoothness on the coefficients of the FBSDE (2.1) (B.4), the\textit{ Four Step Scheme Methodology }ensures existence and uniqueness of solution for (2.2), and if we remove this recquirement of smoothness on the coefficients of the system (2.1), a regularization argument in [5] is used  in order to prove it. In this setting,  we can also conclude the claims (2) and (3) of  \textit{theorem 3.1} and \textit{theorem 4.3}.
\begin{flushleft}
\textbf{Acknowledgments}
\end{flushleft}
The authors wish to thank Fernanda Cipriano for   reading this work, and for her constructive suggestions and questions.

\section*{References} 

\setlength{\parindent}{0cm}  

[1] Azencott R., \textit{Grandes D\'{e}viations et Applications, \'{E}cole d\'{}\'{E}t\'{e} de Probabiliti\'{e}s de Saint-Flour},  1-76, Lecture Notes in Mathematics, vol.774, Springer 1980. 
\newline

[2] Baldi P., Chaleyat-Maurel M., A\textit{n Extension of Ventzell-Freidlin Estimates, Stochastic Analysis and Related Topics}, Silivri 1986, 305-327,  Lecture Notes in Math., 1316, Springer, Berlin, 1988. 
\newline

[3] Bismut J.M.,\textit{ Th\'{e}orie Probabilistique du Contr\^{o}le des Diffusions}, Mem. Amer. Math. Society, 176, Providence, Rhode Island, 1973 
\newline

[4]Crandall M., Ishii H., Lions P.L. ,  \textit{User\`{}s Guide to Viscosity Solutions of Second Order Partial Differential Equations}, Bull. Amer. Math. Soc. 27, 1-67, 1982 
\newline

[5] Delarue F., \textit{On the Existence and Uniqueness of Solutions to FBSDEs in a Non Degenerate Case}, Stoch. Processes and their Applications, 99, 209-286, 2002 
\newline

[6] Delarue F. , \textit{Estimates of the Solutions of a System of Quasi-linear PDEs. A probabilistic scheme}, S\'{e}minaire des Probabilit\'{e}s XXXVII,  290-332,   Lecture Notes in Mathematics 1832, 2003, Springer 
\newline

[7] Dembo A. , Zeitouni O., \textit{Large Deviations Techniques and Applications}, 2nd edition, Jones and Bartlett, 1998
\newline

[8] Doss H. ,  Rainero S., \textit{Sur l\'{}Existence, l\'{}Unicit\'{e}, la Stabilit\'{e} et les Propriet\'{e}t\'{e}s de Grandes D\'{e}viations des Solutions d\'{}\'{E}quations Diff\'{e}rentielles Stochastiques R\'{e}trogrades \`{a} Horizon Al\'{e}atoire. Application \`{a} des Probl\`{e}mes de Perturbations Singuli\`{e}res}, Bull. Sciences Math. 13, 99-174, 2007 
\newline

[9] Essaki H. , \textit{Large Deviation Principle for a Backward stochastic Differential Equation with a Subdifferential Operator}, Comptes Rendus Mathematiques, Acad. Sci. Paris, I-346, 75-78, 2008 
\newline

[10] Frei C., Reis G., \textit{Quadratic FBSDE with Generalized Burgers Type Nonlinearities, PDE Perturbation and Large Deviations}, to appear in Stochastic and Dynamics
\newline

[11] Freidlin M. , Wentzell D.A., \textit{Random Perturbations of Dynamical Systems}, 2nd ed., Springer, 1998 
\newline

[12] Ladyzhenskaja O.A. , Solonnikov V.A., Uraliceva N.N.,\textit{ Linear and Quasilinear Equations of Parabolic Type}, AMS, Providence, RI, 1968
\newline

[13] Lions J.L. , \textit{Quelques M\'{e}thodes de R\'{e}solution aux Probl\`{e}mes Non Lin\'{e}aires}, Dunod,  2002
\newline

[14] Ma J. , Protter P., Young J., \textit{Solving Forward-Backward Stochastic Differential Equations Explicitly - a Four Step Scheme}, Probability Theory and Related Fields, 98, 339-359
\newline

[15] Pardoux E., \textit{Backward Stochastic Differential Equations and Viscosi-ty Solutions of Systems of Semilinear Parabolic and Elliptic PDEs of 2nd order}, Stoch. Analysis and Related Topics, The Geilo Workshop, 79-127,  Birkhauser,1996
\newline

[16] Pardoux E., Tang S., \textit{Forward Backward Stochastic Differential Equations and Quasilinear Parabolic Partial Differential Equations}, Prob. Theory and Related Fields, 114, 123-150,1999
\newline

[17] Peng S., \textit{Probabilistic Interpretation for Systems of Semilinear Parabolic PDEs}, Stochastics and Stochastic Reports 38, 119-139, 1992
\newline

[18] Rainero S., \textit{Un Principe de Grandes D\'{e}viations pour une Equation Diff\'{e}rentielle Stochastique Progressive-R\'{e}tr\'{o}grade}, Comptes Rendus Mathematiques, Acad. Sci. Paris, I-343, 141-144, 2006
\newline

[19]  Thalmaier, A., \textit{On the differentiation of heat semigroups and Poisson integrals}, Stochastics Stochastics Rep., 61, 297-321, 1997
\newline

[20] Thalmaier, A., Wang, F.Y.,\textit{ Gradient estimates for harmonic functions on regular domains in Riemannian manifolds}, J. Funct. Anal., 155, 109-124, 1998 
\newline

[21] Ton, B.A., \textit{Non Stationary Burger Flows with Vanishing Viscosity in Bounded Domains in $\mathbb{R}^3$}, Math. Z., 145, 69-79, 1975
\newline

\vspace{30 pt}

GFMUL and Dep. de Matem\'{a}tica IST(TUL). 

Av. Rovisco Pais 

1049-001 Lisboa, Portugal 

E-mail address: abcruz@math.ist.utl.pt

\vspace{15 pt}

GFMUL - Grupo de F\'{i}sica Matem\'{a}tica 

Universidade de Lisboa 

Avenida Prof. Gama Pinto 2 

1649-003 Lisboa, Portugal

E-mail address:  adeoliveiragomes@sapo.pt

\end{document}